\documentclass[12pt]{amsart}
\usepackage{amssymb}
\usepackage{amsmath, amscd}
\usepackage{amsthm}
\usepackage{comment}
\usepackage[table,xcdraw]{xcolor}
\usepackage{ulem}
\usepackage{tikz}
\usepackage{float}
\usepackage{graphicx}
\usepackage{circuitikz}
\usepackage[colorlinks=true, linkcolor=blue,urlcolor=blue]{hyperref}

\newtheorem{theorem}{Theorem}[section]

\newtheorem{proposition}[theorem]{Proposition}
\newtheorem{lemma}[theorem]{Lemma}

\newtheorem{corollary}[theorem]{Corollary}
\theoremstyle{definition}

\newtheorem{example}[theorem]{Example}
\newtheorem{definition}[theorem]{Definition}

\begin{document}

\author[P. Danchev]{Peter Danchev}
\address{Institute of Mathematics and Informatics, Bulgarian Academy of Sciences, 1113 Sofia, Bulgaria}
\email{danchev@math.bas.bg; pvdanchev@yahoo.com}
	
\author[M. Doostalizadeh]{Mina Doostalizadeh}
\address{Department of Mathematics, Tarbiat Modares University, 14115-111 Tehran Jalal AleAhmad Nasr, Iran}
\email{d\_mina@modares.ac.ir;  m.doostalizadeh@gmail.com}	

\author[A. Moussavi]{Ahmad Moussavi}
\address{Department of Mathematics, Tarbiat Modares University, 14115-111 Tehran Jalal AleAhmad Nasr, Iran}
\email{moussavi.a@modares.ac.ir; moussavi.a@gmail.com}

\title[Weakly Strongly 2-Nil-Clean Rings]{Weakly Strongly 2-Nil-Clean Rings}

\maketitle




\begin{abstract}
In this paper, we introduce and explore in-depth the notion of {\it weakly strongly 2-nil-clean rings} as a common non-trivial generalization of both strongly 2-nil-clean rings and strongly weakly nil-clean rings as defined and studied by Chen-Sheibani in the J. Algebra \& Appl. (2017). We, specifically, succeeded to prove that any weakly strongly 2-nil-clean ring is strongly $\pi$-regular and, concretely, it decomposes as the direct product of a strongly 2-nil-clean ring and a ring of the type $\mathbb{Z}_{2^k}$ for some $k\geq 1$.
\end{abstract}

\section{Introduction and Fundamental Concepts}

Throughout the current article, all rings are assumed to be associative with identity and all modules are unitary. We denote by $U(R)$, $nil(R)$ and $Id(R)$ the sets of invertible elements, nilpotent elements and idempotent elements in $R$, respectively. Likewise, $J(R)$ denotes the Jacobson radical of $R$, and $Z(R)$ denotes the center of $R$. The ring of $n\times n$ matrices over $R$ and the ring of $n\times n$ upper triangular matrices over $R$ are, respectively, denoted by $M_{n}(R)$ and $T_{n}(R)$.

Consulting with \cite{16} and \cite{17}, a ring is called {\it (strongly) clean} if every its element can be written as a sum of a unit and an idempotent (that commute with each other). In the past two decades, there has been intense interest in this class of rings as well as there are numerous generalizations and their induced versions.

In a similar aspect, according to Diesl \cite{2}, a ring is called {\it (strongly) nil-clean} if every its element can be written as a sum of a nilpotent and an idempotent (which commute with each other).

On the other hand, imitating Chen and Sheibani (\cite{chen}), a ring is said to be {\it strongly 2-nil-clean} if every its element is the sum of two idempotents and a nilpotent that commute with each other. They showed that a ring $R$ is strongly 2-nil-clean if, and only if, for all $a \in R$, $a - a^3 \in nil(R)$ if, and only if, for all $a \in R$, there exists an idempotent $e\in R$ such that $a^2 -e\in nil(R)$ if, and only if, every element is the sum of a tripotent and a nilpotent which commute.

Furthermore, again mimicking Chen and Sheibani (\cite{chen2}), a ring is said to be {\it strongly weakly nil-clean} if every its element is the sum or difference of a nilpotent and an idempotent that commutes with each other. They  proved that a ring $R$ is strongly weakly nil-clean if, and only if, for any $a\in R$, there exists an idempotent $e\in Z[a]$ such that $a\pm e \in nil(R)$ if, and only if, $nil(R)$ forms an ideal of $R$ and $R/nil(R)$ is weakly Boolean.

Thus, motivated by all of this, we introduce and investigate the more general concept of weakly strongly 2-nil-clean rings. In fact, we say that a ring $R$ is {\it weakly strongly 2-nil-clean} if, for every element $a\in R$, there are $e,f\in Id(R)$ and $n\in nil(R)$ that commute with each other and $a=\pm e\pm f+n$. We establish that $R$ is weakly strongly 2-nil-clean if, and only if, $30\in nil(R)$ and either $a^3-a\in nil(R)$, or $a(a-1)(a-2)\in nil(R)$, or $a(a+1)(a+2)\in nil(R)$ for all $a\in R$ (see Proposition~\ref{three}). Moreover, as a culmination of our examination, we successfully obtain that $R$ is weakly strongly 2-nil-clean if, and only if, $R=R_1\times  R_2$, where $6\in nil(R)$, $R_1$ is strongly 2-nil-clean and $R_2=\mathbb{Z}_{5^k}$ for some integer $k>1$ (see Theorem~\ref{maj}). In particular, every weakly strongly 2-nil-clean ring is strongly $\pi$-regular (see Lemma~\ref{2.12}).

\section{Results on Weakly Strongly 2-Nil-Clean Rings}

In the present section, we start our serious examination of the so-called weakly strongly 2-nil-clean rings investigating first their elementary properties.

\medskip

For the sake of completeness, we now state our main instrument.

\begin{definition}\label{2.1} Let $R$ be a ring. We say that $R$ is {\it weakly strongly 2-nil-clean} if, for any element $a\in R$, $a=\pm e\pm f+n$ with $e,f,n$ commuting with each other, where $e,f \in Id(R)$ and $n\in nil(R)$.
\end{definition}

Manifestly, strongly 2-nil-clean rings and strongly weakly nil-clean rings are both weakly strongly 2-nil-clean. However, the next construction demonstrates that the converse implications are impossible.

\begin{example} Let $R=\mathbb{Z}_5$. Then, a simple check shows that $R$ is weakly strongly 2-nil-clean, but it is {\it not} strongly weakly nil-clean, because $3$ can not expressed as $\pm e+n$, where $e=e^2\in R$ and $n\in nil(R)$. Also, we simply see that $R$ is {\it not} strongly 2-nil-clean, because $3$ can not expressed as $e+f+n$, where $e,f\in Id(R)$ and $n\in nil(R)$.
\end{example}

We now begin with a series of technicalities.

\begin{lemma}\label{1} Let $k\leq 2$ be an integer. Then, $\mathbb{Z}_{5^k}$ is a weakly strongly $2$-nil-clean ring.
\end{lemma}

\begin{proof} It is straightforward, so we leave the inspection to the interested reader.
\end{proof}

\begin{proposition}\label{hom} {\rm (i)} Any homomorphic image of a weakly strongly 2-nil-clean ring is again a weakly strongly 2-nil-clean ring, but the reciprocal claim is not generally true. In particular, $R$ is a weakly strongly 2-nil-clean ring if, and only if, the quotient $R/I$ is a weakly strongly 2-nil-clean ring, whenever $I$ is a nil-ideal of $R$.\\\indent
{\rm (ii)} The class of weakly strongly 2-nil-clean rings is not closed under formation of finite products; e.g., $\mathbb{Z}_5 \times \mathbb{Z}_5$ is not weakly strongly 2-nil-clean.
 \end{proposition}

\begin{proof} (i) Apparently, any homomorphic image of a weakly strongly 2-nil-clean ring is weakly strongly 2-nil-clean, and the opposite implication is wrong as, for example, $\mathbb{Z}/5\mathbb{Z}\simeq \mathbb{Z}_5$ is a weakly strongly 2-nil-clean ring, but $\mathbb{Z}$ is obviously {\it not} a weakly strongly 2-nil-clean ring.

Now, assume that $I$ is a nil-ideal of $R$ and the factor-ring $R/I$ is weakly strongly 2-nil-clean. Choosing $a\in R$, we have $a+I\in R/I$. Then, there are $e+I,f+I\in Id(R/I)$ and $n+I\in nil(R/I)$ commuting with each other such that $a+I=\pm e\pm f+n+I$. As all idempotents lift modulo $I$, there are $g,h\in R$ such that $g-e\in nil(R)$ and $h-f\in nil(R)$, as required.

(ii) can directly be checked for validness.
\end{proof}

\begin{lemma}\label{2} Let $R$ be a weakly strongly 2-nil-clean ring. If $2\in nil(R)$, then $R$ is strongly nil-clean.
\end{lemma}

\begin{proof} Choose $r\in R$. Then, there are idempotents $e,f\in R$ and a nilpotent $n\in R$ such that $r=\pm e\pm f+n$ such that $e,f$ and $n$ commute with each other. As $2\in nil(R)$, we get $e=-e+m$, where $m\in nil(R)$ and $f=-f+q$, where $q\in nil(R)$. It, thus, follows that $r=e+f+k$, where $k\in nil(R)$ and $e,f$ and $k$ commute with each other.

On the other hand, $(e+f)^2=e+f+2ef$ implying $(e+f)^2-(e+f)\in nil(R)$. Exploiting \cite[Lemma 2.4]{zhou}, there exists $\alpha(t)\in \mathbb{Z}$ such that $[\alpha(e+f)]^2=\alpha(e+f)$ and $\alpha(e+f)=(e+f)+p$, where $p\in nil(R)$. It, thereby, follows that $r=\alpha(e+f)-p+q$ such that $-p+q\in nil(R)$ and $\alpha(e+f)(-p+q)=(-p+q)\alpha(e+f)$. So, $R$ is strongly nil-clean, as expected.
\end{proof}

\begin{lemma}\label{3} Let $R$ be a weakly strongly 2-nil-clean ring. If $3\in nil(R)$, then $R$ is strongly 2-nil-clean.
\end{lemma}

\begin{proof} Given $r\in R$, there are idempotents $e,f\in R$ and a nilpotent $n\in R$ such that $r=\pm e\pm f+n$, and $e,f$ and $n$ commute with each other. If, for a moment, $r=e+f+n$, then $r$ has strongly $2$-nil-clean expression and so there is nothing to prove.

If, however, $r=e-f+n$, since $3\in nil(R)$ we get $r^3-r\in nil(R)$. As $2\in U(R)$, \cite[Lemma 2.6]{zhou} applies to deduce that there exists $\alpha(t)\in \mathbb{Z}$ such that $\alpha(r)^3=\alpha(r)$ and $r=\alpha(r)+p$, where $p\in nil(R)$ and $p\alpha(r)=\alpha(r)p$. So, $r$ has the desired strongly 2-nil-clean expression in this case.

If now $r=-e-f+n$, then from $3\in nil(R)$ we derive $r^3-r\in nil(R)$. Utilizing a similar argument as that of above, there exists $\beta(t)\in \mathbb{Z}$ such that $\beta(r)^3=\beta(r)$ and $r=\beta(r)+q$, where $q\in nil(R)$ and $q\beta(r)=\beta(r)q$. Consequently, $r$ has a strongly 2-nil-clean expression and, therefore, $R$ is strongly 2-nil-clean, as suspected.
\end{proof}

As a valuable consequence, we yield.

\begin{corollary}\label{4} Let $R$ be a weakly strongly 2-nil-clean ring. If $6\in nil(R)$, then $R$ is strongly 2-nil-clean.
\end{corollary}

\begin{proof} As $6\in nil(R)$, we may write $R/6R\simeq R_1\times R_2$, where $2\in nil(R_1)$ and $3\in nil(R_2)$. But, as $R_1$ and $R_2$ are weakly strongly 2-nil-clean rings in conjunction with Lemmas \ref{2} and \ref{3}, one concludes that $R/6R$ is too a strongly 2-nil-clean ring. Thus, $R$ is a strongly 2-nil-clean ring, as wanted.
\end{proof}

\begin{lemma}\label{5} Let $R$ be a weakly strongly 2-nil-clean ring. If there are $e,f\in Id(R)$ and $n\in nil(R)$ commuting with each other such that  $-2=e\pm f+n$, then $R$ is strongly 2-nil-clean.
\end{lemma}

\begin{proof} Write $-2=e+f+n$, where $e,f\in Id(R)$ and $n\in nil(R)$ such that they commute with each other. Then, $$4=e+f+2ef+m=-2+2ef+m,$$ where $m=n^2+2en+2fn\in nil(R)$. So, $6=2ef+m$ and hence $36=4ef+q$, where $q=m^2+2efm\in nil(R)$. From $6=2ef+m$, we get $12=4ef+2m$. It follows that $$36-12=4ef-4ef+q-2m=q-2m\in nil(R).$$ Thus, $6\in nil(R)$ and so Corollary \ref{4} works to derive that $R$ is strongly 2-nil-clean.

However, if $-2=e-f+n$, then one writes that $$4=e+f-2ef+m=e-f+2f-2ef+m=-2+2f-ef+m$$ for some $m\in nil(R)$. Thus, $6=2f-2ef+m$, forcing that $$6e=2ef-2ef+me=me\in nil(R).$$ Further, from $-2=e-f+n$, we get $-12=6e-6f+6n=-6f+6n$. Hence, $12=6f-6n$. From this and $6=2f-2ef+m$, we infer $$18=6f-6ef+3m=6f+3m=12+6n+3m.$$ It, therefore, follows that $6\in nil(R)$ and Corollary \ref{4} guarantees that $R$ is strongly 2-nil-clean, as promised.
\end{proof}

\begin{lemma}\label{2nil} Let $R$ be a ring such that $3\in nil(R)$. Then, the following three statements are equivalent:\\\indent (1) $R$ is a strongly 2-nil-clean ring;\\\indent (2) For any $a \in R$, there exist two idempotents $e, f \in R$ and a nilpotent $w\in  R$ which commute with each other such that $a = e - f + w$;\\\indent  (3) For any $a \in R$, there exist two idempotents $e, f \in R$ and a nilpotent $w\in  R$ which commute with each other such that $a = -e - f + w$.
\end{lemma}

\begin{proof} (1) $\Leftrightarrow$ (2). It follows at once from \cite[Lemma 2.2]{chen}.\\
(2) $\Rightarrow$ (3). Assume that $a\in R$. Then, $a=e-f+n$, where $e,f\in Id(R)$ and $n\in nil(R)$ that all commute. Since $a=ea+(1-e)a$, we get $a=e(1-f)-(1-e)f+n$. As $3\in nil(R)$, we receive $3e(1-f)\in nil(R)$ and so $e(1-f)=-2e(1-f)+m$, where $m=3e(1-f)$. It, consequently, follows that $$a=-2e(1-f)-(1-e)f+n+m=-e(1-f)-(e(1-f)+(1-e)f).$$ It is now easy to verify that $e(1-f), (e(1-f)+(1-e)f)\in Id(M)$, as requested.\\\indent (3) $\Rightarrow$ (2). Letting $a\in R$, then $a-1=-e-f+n$, where $e,f\in Id(R)$ and $n\in nil(R)$ that all commute. So, $a=1-e-f+n$, as needed.
\end{proof}

We now have all the machinery necessary to illustrate truthfulness of the following main assertion.

\begin{proposition}\label{2.2} Let $\{R_{\alpha}\}$ be a finite collection of rings. Then, the direct product
$\prod_{\alpha} R_{\alpha}$ is weakly strongly 2-nil-clean if, and only if, each $R_{\alpha}$ is weakly strongly 2-nil-clean, and at most one $R_{\alpha}$ is not strongly 2-nil-clean.
\end{proposition}

\begin{proof} Let $\prod_{\alpha} R_{\alpha}$ be weakly 2-nil-clean. Then, each $R_{\alpha}$ is weakly 2-nil-clean, because, in view of Proposition~\ref{hom}, homomorphic images of weakly strongly 2-nil-clean rings are again weakly strongly 2-nil-clean.

Suppose now that there are two indices, say $\alpha_1$ and $\alpha_2$, such that $\alpha_1\neq \alpha_2$ and both $R_{\alpha_1}$ and $R_{\alpha_2}$ are {\it not} 2-nil-clean. But, as $R_{\alpha_1}\times R_{\alpha_2}$ is a homomorphic image of $R = \prod R_{\alpha}$, we can get $R_{\alpha_1}\times R_{\alpha_2}$ is weakly strongly 2-nil-clean. Considering the element $(-2,2)$ of $R_{\alpha_1}\times R_{\alpha_2}$, one finds that there are $(e,f),(g,h)\in Id(R_{\alpha_1}\times R_{\alpha_2})$ and $(n_1,n_2)\in nil(R_{\alpha_1}\times R_{\alpha_2})$ commuting with each other such that $(-2,2)=\pm(e,f)\pm (g,h)+(n_1,n_2)$. If $(-2,2)=\pm(e,f)+ (g,h)+(n_1,n_2)$, then $-2=\pm e+g+n_1$. So, Lemma \ref{5} gives that $R_{\alpha_1}$ is strongly 2-nil-clean -- a contradiction. If $(-2,2)=(e,f)\pm (g,h)+(n_1,n_2)$, then via a similar argument $R_{\alpha_1}$ is strongly 2-nil-clean -- again a contradiction. Finally, if $(-2,2)=-(e,f)- (g,h)+(n_1,n_2)$, then $2=-f-h+n_2$ and so $-2=f+h-n_2$. Applying Lemma \ref{5}, it must be that $R_{\alpha_2}$ is strongly 2-nil-clean, which is an absurd. Thus, $(-2,2)$ can not express as weakly strongly 2-nil-clean decomposition. So, either $\alpha_1$ or $\alpha_2 $ has to be strongly 2-nil-clean, as stated.

Conversely, assume that $R_{{\alpha}_1}$ is a weakly strongly 2-nil-clean ring that is {\it not} strongly 2-nil-clean and $R_{\alpha}$ is strongly 2-nil-clean for all other $\alpha\neq \alpha_1$. We can now write $\prod_{\alpha\neq \alpha_1} R_{\alpha}\simeq R_1\times R_2$, where $3\in nil(R_1)$ and $2\in nil(R_2)$. So, we arrive at the isomorphism $\prod_{\alpha} R_{\alpha}\cong R_{\alpha_1}\times R_1\times R_2$, where, as proven above in Lemmas~\ref{3} and \ref{2}, $R_{\alpha_1}$ is weakly strongly 2-nil-clean that is {\it not} strongly 2-nil-clean, $R_1$ is a strongly 2-nil-clean ring and $R_2$ is a strongly nil-clean ring. Now, let $(r,r_1,r_2)\in\prod R_{\alpha}$. Then, there are $e,f\in Id(R)$ and $n\in nil(R)$ all commuting such that $r=\pm e\pm f+n$. We, hereafter, differ three cases as follows:

\medskip

\noindent{\bf Case (1):} If $r=e+f+n$, then Lemma \ref{2nil} enables us that $r_1=e_1+f_1+n_1$ for some $e_1,f_1\in Id(R_1)$ and $n_1\in nil(R_1)$ which all commute, and that there exists $e_2\in Id(R_2)$ and $n_2\in nil(R_2)$ which commute such that $r_2=e_2+n_2$. It follows now that $$(r,r_1,r_2)=(e,e_1,e_2)+(f,f_1,0)+(n,n_1,n_2),$$ and, surely, $(e,e_1,e_2),(f,f_1,0)\in Id(\prod_{\alpha} R_{\alpha})$ and $(n,n_1,n_2)\in nil(\prod_{\alpha} R_{\alpha})$ that commute with each other.

\medskip

\noindent{\bf Case (2):} If $r=e-f+n$, then Lemma \ref{2nil} tells us that $r_1=e_1-f_1+n_1$ for some $e_1,f_1\in Id(R_1)$ and $n_1\in nil(R_1)$ which all commute, and that there exists $e_2\in Id(R_2)$ and $n_2\in nil(R_2)$ which commute such that $r_2=e_2+n_2$. It follows now that $$(r,r_1,r_2)=(e,e_1,e_2)-(f,f_1,0)+(n,n_1,n_2),$$ and, evidently, $(e,e_1,e_2),(f,f_1,0)\in Id(\prod_{\alpha} R_{\alpha})$ and $(n,n_1,n_2)\in nil(\prod_{\alpha} R_{\alpha})$ that commute with each other.

\medskip

\noindent{\bf Case (3):} If $r=-e-f+n$, then Lemma \ref{2nil} teaches us that $r_1=-e_1-f_1+n_1$ for some $e_1,f_1\in Id(R_1)$ and $n_1\in nil(R_1)$ which all commute, and that there exists $e_2\in Id(R_2)$ and $n_2\in nil(R_2)$ which commute such that $r_2=e_2+n_2$, whence $r_2=-e_2+n_2+2e_2$ and $n_2+2e_2\in nil(R_2)$. It follows now that $$(r,r_1,r_2)=-(e,e_1,e_2)-(f,f_1,0)+(n,n_1,n_2+2e_2),$$ and, easily, $(e,e_1,e_2),(f,f_1,0)\in Id(\prod_{\alpha} R_{\alpha})$ and $(n,n_1,n_2+2e_2)\in nil(\prod_{\alpha} R_{\alpha})$ that commute with each other, thus finishing the entire proof.
\end{proof}

As an immediate consequence, we derive:

\begin{corollary}\label{cor} Let $S$ be a strongly 2-nil-clean ring. Then, $R\times S$ is a weakly strongly 2-nil-clean ring if, and only if, so is $R$.
\end{corollary}

We now proceed by proving the following technical statement.

\begin{lemma}\label{6}
Let $R$ be a weakly strongly 2-nil-clean ring. Then, $30\in nil(R)$.
\end{lemma}

\begin{proof} As $R$ is weakly strongly 2-nil-clean, there are elements $e,f\in Id(R)$ and $m\in nil(R)$ that commute with each other and $3=\pm e\pm f+n$.

Firstly, if $3=e+f+n$, then $f=3-e-n$. So, $3-e-n=(3-e-n)^2$ ensuring that $6-4e\in nil(R)$ and so $2e=e(6-4e)\in nil(R)$. Similarly, we can elementarily see that $2f\in nil(R)$. From this and $3\times 2=2e+2f+2n$, we get $6\in nil(R)$ and hence $30\in nil(R)$.

Next, if $3=-e-f+n$, then $f=-3-e+n$. So, $-3-e+n=(-3-e+n)^2$ assuring that $8e+12\in nil(R)$ and so $20e=e(8e+12)\in nil(R)$. Analogously, we get $20f\in nil(R)$. Multiplying the equation $3=-e-f+n$ by $20$, we obtain $30\in nil(R)$, as needed.

Finally, if $3=e-f+n$, then $f=e-3+n$ and hence $e-3+n=(e-3+n)^2$. It follows that $6e-12\in nil(R)$. From this and $6\times 3=6e-6f+6n$, we receive $6+6f\in nil(R)$ and so $12f=f(6+6f)\in nil(R)$. Since $6e-12\in nil(R)$, it must be that $6e=-e(6e-12)\in nil(R)$. Thus, $12\times 3=12e-12f+12n$ insures that $36\in nil(R)$. Therefore, $6\in nil(R)$ whence $30\in nil(R)$, as pursued.
\end{proof}

\begin{lemma}\label{j}
Let $R$ be a weakly strongly 2-nil-clean ring. Then, $J(R)$ is a nil-ideal.
\end{lemma}	

\begin{proof} Assuming that $j\in J(R)$, there are elements $e,f\in Id(R)$ and $n\in nil(R)$ that commute with each other and $j=\pm e\pm f+n$. Henceforth, we break the proof into three distinguished cases:

\medskip

\noindent{\bf Case (1):} If $j=e+f+n$, then $j-e-n=f$ and so $(j-e-n)^2=f$. It follows that $j^2-2ej+e=f+m$, where $m=-n(n+2e+2j)\in nil(R)$. Hence, $j(j-2e)=f-e+m$ and, since $me=em$, $mf=fm$, we obtain that $$j^3(j-2e)^3=(j(j-2e))^3=f-e+k,$$ where $$k=m(3(e-f)^2+3m(e-f)+m^2)\in nil(R).$$ Thus, $j^3(j-2e)^3-j(j-2e)=k-m\in nil(R)$. This means that $j(j-2e)(j^2(j-2e)^2-1)\in nil(R)$ and, since $j^2(j-2e)^2-1\in U(R)$, we have $j(j-2e)\in nil(R)$. So, $f-e+m\in nil(R)$ and hence $f-e\in nil(R)$. Therefore, $f=e+q$, where $q\in nil(R)$ and $eq=qe$, $fq=qf$ and $qn=nq$. From this and $j=e+f+n$, we readily get $j=2e+q+n$. As $30\in nil(R$), Lemma \ref{6} can be applied to get that $j^5=32e+t$ for some $t\in nil(R)$. Consequently, $j^5=2+t$ whence $j^5-j=t-q-n\in nil(R)$. But, as $j^4-1\in U(R)$, we conclude $j\in nil(R)$, as asked.

\medskip

\noindent{\bf Case (2):} If $j=-e-f+n$, then $f=-j-e+n$ and so $$-j-e+n=(-j-e+n)^2=j^2+e+2je+m,$$ where $m=n^2-2en-2jn\in nil(R)$. It follows that $-j-2e-j^2-2je=m-n\in nil(R)$. So, $(1+j)(j+2e)\in nil(R)$. As $(1+j)\in U(R)$, we detect that $j+2e\in nil(R)$. Hence, $j=-2e+k$, where $k=j+2e\in nil(R)$ and, clearly, $ke=ek$. As $30\in nil(R)$, Lemma \ref{6} employs to get $j^5=-32e+q$, where $q\in nil(R)$ and $qe=eq$. Therefore, $j^5=-2e+q-30e$. So,  $$j(j^4-1)=j^5-j=q+30e-k\in nil(R).$$ But, as $j^4-1\in U(R)$, we infer $j\in nil(R)$, as asked.

\medskip

\noindent{\bf Case (3):} If $j=e-f+n$, then $e=j+f-n$ and thus $$j+f-n=(j+f-n)^2=j^2+f+2jf+m,$$ where $m=n^2-2fn-2jn\in nil(R)$. It follows that $j-j^2-2jf=m+n\in nil(R)$. Thus, we discover that $$-f(j-j^2-2jf)=j^2f+jf=(j+1)jf\in nil(R).$$ As $1+jf\in U(R)$, we know $jf\in nil(R)$. From this and $j-j^2=2jf+m+n\in nil(R)$, we derive $j(1-j) \in nil(R)$. Furthermore, since $1-j\in U(R)$, we extract $j\in nil(R)$. So, in either case, $J(R)\subseteq Nil(R)$, as asked for.
\end{proof}

Standardly, a ring $R$ is said to be {\it strongly $\pi$-regular}, provided that, for any $a \in R$, there exists $n \in \mathbb{N}$ such that $a^n \in a^{n+1}R$.

\medskip

The following claim is crucial.

\begin{lemma}\label{2.12}
Every weakly strongly 2-nil-clean ring is strongly $\pi$-regular.
\end{lemma}

\begin{proof} Let $R$ be a weakly strongly 2-nil-clean ring, and set $a \in R$. Then, there are $e,f\in Id(R)$ and $n\in nil(R)$ that commute each with other and $a=\pm e\pm f+n$. We now claim that $a^5-a\in nil(R)$.

In fact, if $a=e+f+n$, then one writes that $$a^5=e+5ef+10ef+10ef+5ef+f+m,$$ where $m\in nil(R)$. So, we have $a^5=e+f+m+30ef$ and, since $30\in nil(R)$, we get $a^5-a=m+30ef-n\in nil(R)$. If now $a=e-f+n$, then $$a^5=e-5ef+10ef-10ef+5ef-f+k,$$ where $k\in nil(R)$. So, $a^5=e-f+k$ and thus $a^5-a=k-n\in nil(R)$. If, finally, $a=-e-f+n$, we arrive at $a^5=-e-f-30ef+q$, where $q\in nil(R)$. Therefore, $a^5-a\in nil(R)$ in either case, as claimed. Hence, there exists $t \in \mathbb{N}$ such that $(a^5-a)^t = 0$. It now follows that $a^{t} = a^{t+1}r$ for some $r \in R$. Consequently, $R$ is a strongly $\pi$-regular ring, as asserted.
\end{proof}	

We continue our work by asking validity of the following useful criterion.

\begin{proposition}\label{three} Let $R$ be a ring. Then, $R$ is weakly strongly 2-nil-clean if, and only if, $30\in nil(R)$ and, either $a^3-a\in nil(R)$, or $a(a-1)(a-2)\in nil(R)$, or $a(a+1)(a+2)\in nil(R)$ for all $a\in R$.
\end{proposition}

\begin{proof} Assuming that $R$ is weakly strongly 2-nil-clean, then Lemma \ref{6} is applicable to get that $30\in nil(R)$. Let $a\in R$. Then, $a=\pm e\pm f+n$, where $e,f\in Id(R)$ and $n\in nil(R)$ such that they commute with each other. Firstly, if $a=e-f+n$ or $a=-e+f+n$, then it is easy to see $a^3-a\in nil(R)$.

Next, if $a=e+f+n$, then $1-a=1-e-f-n$ and one can plainly check that $(1-a)^3-(1-a)\in nil(R)$. It now follows that $a^3-3a^2+2a\in nil(R)$ and so $a(a-1)(a-2)\in nil(R)$.

If now $a=-e-f+n$, then $1=a=1-e-f+n$ and thus $(1+a)^3-(1+a)\in nil(R)$ guaranteeing that $a^3+3a^2+2a\in nil(R)$.  Therefore, $a(a+1)(a+2)\in nil(R)$, as stated.

Conversely, assume that $a\in R$. As $30\in nil(R)$, we may decompose $R\simeq R_1\times R_2$, where $15\in nil(R_1)$ and $2\in nil(R_1)$. Then, we can write $a=(a_1,a_2)$, where $a_1\in R_1$ and $r_2\in R_2$.

If, first, $a^3-a\in nil(R)$, then we get $a_i^3-a_i\in nil(R)$ with $i=1,2$. As $2\in U(R_1)$,
\cite[Lemma 2.6]{zhou} insures that there exists $\alpha[t]\in \mathbb{Z}[t]$ such that $\alpha(a_1)^3=\alpha(a_1)$ and $a_1-\alpha(a_1)\in nil(R_1)$. Setting $e_1:=(\alpha(a_1)^2+\alpha(a_1))/2,f_1:=(\alpha(a_1)^2-\alpha(a_1))/2\in nil(R_1)$, we write $a_1=e_1-f_1+m_1$, where $m_1=a_1-\alpha(a_1)\in nil(R_1)$. Since $2\in nil(R)$, Lemma \ref{2} gives the existence of an idempotent $e_2\in R_2$ such that $a_2-e_2\in nil(R)$. It now follows that $a_2=e_2+m_2$, where $m_2 \in nil(R_2)$. Thus, $a=(e_1,e_2)-(f_1,0)+(m_1,m_2)$, as required.

If next $a(a-1)(a-2)\in nil(R)$, then we write $$(a-1)(a^2-2a)=(a-1)(a^2-2a+1-1)=(a-1)((a-1)^2-1)=(a-1)^3-(a-1)\in nil(R).$$ So, $(a_i-1)^3-(a_i-1)\in nil(R)$ for $i=1,2$. As $2\in U(R_1)$, again appealing to \cite[lemma 2.6]{zhou}, there exists $\alpha[t]\in \mathbb{Z}[t]$ such that $\alpha(a_1-1)^3=\alpha(a_1-1)$ and $a_1-1-\alpha(a_1-1)\in nil(R_1)$. Take $e_1:=(\alpha(a_1-1)^2+\alpha(a_1-1))/2,f_1:=(\alpha(a_1-1)^2-\alpha(a_1-1))/2\in nil(R_1)$, we write $a_1-1=e_1-f_1+m_1$, where $m_1=a_1-1-\alpha(a_1-1)\in nil(R_1)$. So, $a_1=e_1+1-f_1+m_1$. Since $2\in nil(R)$, owing to Lemma \ref{2}, there exists an idempotent $e_2\in R_2$ such that $a_2-e_2\in nil(R)$. It, therefore, follows that $a_2=e_2+m_2$, where $m_2 \in nil(R_2)$. Thus, $a=(e_1,e_2)+(1-f_1,0)+(m_1,m_2)$, as required.

If now $a(a+1)(a+2)\in nil(R)$, then one verifies that $$(a+1)(a^2+2a+1-1)=(a+1)^3-(a+1)\in nil(R).$$ So, $(a_i-1)^3-(a_i-1)\in nil(R)$ whenever $i=1,2$. As $2\in U(R_1)$, we once again may apply \cite[Lemma 2.6]{zhou} to get the existence of $\alpha[t]\in \mathbb{Z}[t]$ such that $\alpha(a_1+1)^3=\alpha(a_1+1)$ and $a_1+1-\alpha(a_1+1)\in nil(R_1)$. Putting $e_1:=(\alpha(a_1+1)^2+\alpha(a_1+1))/2,f_1:=(\alpha(a_1+1)^2-\alpha(a_1+1))/2\in nil(R_1)$, we write $a_1+1=e_1-f_1+m_1$, where $m_1=a_1+1-\alpha(a_1+1)\in nil(R_1)$. So, $a_1=-1+e_1-f_1+m_1$. Since $2\in nil(R)$, the application of Lemma \ref{2} leads to the existence of an idempotent $e_2\in R_2$ such that $a_2-e_2=a_2+e_2-2e_2\in nil(R)$. It, thereby, follows that $a_2=-e_2+m_2$, where $m_2 \in nil(R_2)$. Thus, $a=-(1-e_1,e_2)-(f_1,0)+(m_1,m_2)$, as required, completing the proof.
\end{proof}

The following necessary and sufficient condition is helpful.

\begin{proposition}\label{11} Let $I$ be a nil-ideal of $R$. Then, $R$ is weakly strongly 2-nil-clean if, and only if, $R/I$ is weakly strongly 2-nil-clean.
\end{proposition}

\begin{proof} One implication is quite clear. Assume now that $R/I$ is weakly strongly 2-nil-clean. Let $a\in R$. Then, with Lemma \ref{three} in hand, we get either $a^3-a+I\in nil(R/I)$, or $a(a-1)(a-2)+I\in nil(R/I)$, or $a(a+1)(a+2)+I\in nil(R/I)$. As $I$ is a nil-ideal of $R$, we arrive at either $a^3-a\in nil(R)$, or $a(a-1)(a-2)\in nil(R)$, or $a(a+1)(a+2)\in nil(R)$ for all $a\in R$. Hence, $R$ is weakly strongly 2-nil-clean with the aid of Lemma~\ref{three}.
\end{proof}

As two pivotal reduction consequences, we derive:

\begin{corollary}\label{fac} Let $R$ be a ring. Then, $R$ is weakly strongly 2-nil-clean if, and only if, $J(R)$ is a nil-ideal and $R/J(R)$ is weakly strongly 2-nil-clean.
\end{corollary}

\begin{proof} It follows directly from Lemmas~ \ref{j} and \ref{11}.
\end{proof}

\begin{corollary}\label{10}
Let $I$ be an ideal of a ring $R$. Then, the following are equivalent:
\begin{enumerate}
\item
$R/I$ is weakly strongly 2-nil-clean.
\item
$R/I^n$	is weakly strongly 2-nil-clean for all $n \in \mathbb{N}$.
\item
$R/I^n$ is weakly strongly 2-nil-clean for some $n \in \mathbb{N}$.
\end{enumerate}	
\end{corollary}

\begin{proof}
(i) $\Longrightarrow$ (ii). For any $n \in \mathbb{N}$, one knows that $$\dfrac{R/I^n}{I/I^n} \cong R/I.$$ Since $I/I^n$ is a nil-ideal of $R/I^n$ and $R/I$ is weakly strongly 2-nil-clean, then Proposition \ref {11} allows us to infer that $R/I^n$ is a weakly strongly 2-nil-clean ring.

(ii) $\Longrightarrow$ (iii). This is pretty trivial.

(iii) $\Longrightarrow$ (i). For any ideal $I$ of $R$, one has that $$\dfrac{R/I^n}{I/I^n} \cong R/I.$$ So, invoking Proposition \ref{11}, we conclude that $R/I$ is weakly strongly 2-nil-clean.
\end{proof}

We are now able to prove the following.

\begin{proposition}\label{12}
Let $R$ be a ring. Then, the following are equivalent:
\begin{enumerate}
\item
$R$ is strongly 2-nil-clean.
\item
${\rm T}_{n}(R)$ is strongly 2-nil-clean for all $n\in \mathbb{N}$.
\item
${\rm T}_n(R)$ is weakly strongly 2-nil-clean for some $n\geq 2$.
\end{enumerate}
\end{proposition}

\begin{proof}
(ii) $\Rightarrow$ (iii). This is elementary.

(iii) $\Rightarrow$ (i). From Corollary \ref{fac}, $J(T_n(R))$ is a nil-ideal of $T_n(R)$ and $T_n(R)/J(T_n(R))$ is weakly strongly 2-nil-clean. But, as $$T_n(R)/J(T_n(R))\simeq \prod_{i=1}^n R/J(R),$$ then using again Corollary \ref{fac}, the quotient-ring $R/J(R)$ is strongly 2-nil-clean. Since $J(R)$ is a nil-ideal of $R$, we obtain that $R$ is strongly 2-nil-clean in virtue of \cite{chen}.

(i) $\Rightarrow$ (ii). It follows from Lemma~\ref{j} that $J(R)$ is a nil-ideal of $R$ and $\prod_{i=1}^n R/J(R)$ is strongly 2-nil-clean. However, since $$T_n(R)/J(T_n(R)) \cong \prod_{i=1}^nR/J(R),$$ we have $T_n(R)/J(T_n(R))$ is strongly 2-nil-clean. Moreover, since $J(T_n(R))$ is nil, we see that $T_n(R)$ is strongly 2-nil-clean, and so weakly strongly 2-nil-clean, as formulated.
\end{proof}

Let $R$ be a ring and $M$ a bi-module over $R$. The trivial extension of $R$ and $M$ is defined as
\[ T(R, M) = \{(r, m) : r \in R \text{ and } m \in M\}, \]
with addition defined component-wise and multiplication defined by
\[ (r, m)(s, n) = (rs, rn + ms). \]

We now are ready to establish the following.

\begin{corollary}\label{2.17}
Let $R$ be a ring and $M$ a bi-module over $R$. Then, the following two items are equivalent:
\begin{enumerate}
\item
$T(R, M)$ is a weakly strongly 2-nil-clean ring.
\item
$R$ is a weakly strongly 2-nil-clean ring.
\end{enumerate}
\end{corollary}

\begin{proof}
Set $A:={\rm T}(R, M)$ and consider $I:={\rm T}(0, M)$. It is not too hard to verify that $I$ is a nil-ideal of $A$ such that $\dfrac{A}{I} \cong R$. So, the result follows immediately from Proposition \ref{11}, as needed.
\end{proof}

Let $\alpha$ be an endomorphism of a ring $R$ and $n$ a positive integer. It was defined by Nasr-Isfahani in \cite{3} the {\it skew triangular matrix ring} like this:
$${\rm T}_{n}(R,\alpha )=\left\{ \left. \begin{pmatrix}
	a_{0} & a_{1} & a_{2} & \cdots & a_{n-1} \\
	0 & a_{0} & a_{1} & \cdots & a_{n-2} \\
	0 & 0 & a_{0} & \cdots & a_{n-3} \\
	\ddots & \ddots & \ddots & \vdots & \ddots \\
	0 & 0 & 0 & \cdots & a_{0}
\end{pmatrix} \right| a_{i}\in R \right\}$$
with addition point-wise and multiplication given by:
\begin{align*}
&\begin{pmatrix}
		a_{0} & a_{1} & a_{2} & \cdots & a_{n-1} \\
		0 & a_{0} & a_{1} & \cdots & a_{n-2} \\
		0 & 0 & a_{0} & \cdots & a_{n-3} \\
		\ddots & \ddots & \ddots & \vdots & \ddots \\
		0 & 0 & 0 & \cdots & a_{0}
	\end{pmatrix}\begin{pmatrix}
		b_{0} & b_{1} & b_{2} & \cdots & b_{n-1} \\
		0 & b_{0} & b_{1} & \cdots & b_{n-2} \\
		0 & 0 & b_{0} & \cdots & b_{n-3} \\
		\ddots & \ddots & \ddots & \vdots & \ddots \\
		0 & 0 & 0 & \cdots & b_{0}
	\end{pmatrix}  =\\
	& \begin{pmatrix}
		c_{0} & c_{1} & c_{2} & \cdots & c_{n-1} \\
		0 & c_{0} & c_{1} & \cdots & c_{n-2} \\
		0 & 0 & c_{0} & \cdots & c_{n-3} \\
		\ddots & \ddots & \ddots & \vdots & \ddots \\
		0 & 0 & 0 & \cdots & c_{0}
\end{pmatrix},
\end{align*}
where $$c_{i}=a_{0}\alpha^{0}(b_{i})+a_{1}\alpha^{1}(b_{i-1})+\cdots +a_{i}\alpha^{i}(b_{0}),~~ 1\leq i\leq n-1
.$$ We designate the elements of ${\rm T}_{n}(R, \alpha)$ by $(a_{0},a_{1},\ldots , a_{n-1})$. If $\alpha $ is the identity endomorphism, then ${\rm T}_{n}(R,\alpha)$ is a subring of upper triangular matrix ring ${\rm T}_{n}(R)$.

\medskip

We are now in a position to show that the following is fulfilled.

\begin{corollary}\label{2.20}
Let $R$ be a ring. Then, the following two issues are equivalent:
\begin{enumerate}
\item
${\rm T}_{n}(R,\alpha)$ is a weakly strongly 2-nil-clean ring.
\item
$R$ is a weakly strongly 2-nil-clean ring.
\end{enumerate}
\end{corollary}

\begin{proof}
Set
$$I:=\left\{
\left.
\begin{pmatrix}
		0 & a_{12} & \ldots & a_{1n} \\
		0 & 0 & \ldots & a_{2n} \\
		\vdots & \vdots & \ddots & \vdots \\
		0 & 0 & \ldots & 0
	\end{pmatrix} \right| a_{ij}\in R \quad (i\leq j )
	\right\}.$$
Then, one readily inspects that $I^{n}=(0)$ as well as that $$\dfrac{{\rm T}_{n}(R,\alpha )}{I} \cong R.$$ Consequently, Proposition \ref{11} works to obtain the pursued conclusion.
\end{proof}

Let $\alpha$ be an endomorphism of a ring $R$. We design by $R[x,\alpha ]$ the {\it skew polynomial ring} whose elements are the polynomials over $R$ as the addition is defined traditionally, and the multiplication is defined by the equality $xr=\alpha (r)x$ for any $r\in R$. Thus, there is an obvious ring isomorphism $$\varphi : \dfrac{R[x,\alpha]}{\langle x^{n}\rangle }\rightarrow {\rm T}_{n}(R,\alpha),$$ given by $$\varphi (a_{0}+a_{1}x+\ldots +a_{n-1}x^{n-1}+\langle x^{n} \rangle )=(a_{0},a_{1},\ldots ,a_{n-1})$$ with $a_{i}\in R$, $0\leq i\leq n-1$. So, one detects that there is an isomorphism $${\rm T}_{n}(R,\alpha )\cong \dfrac{R[x,\alpha ]}{\langle  x^{n}\rangle}$$ of rings, where $\langle x^{n}\rangle$ is the ideal generated by $x^{n}$. 

Likewise, let us denote by $R[[x, \alpha]]$ the ring of {\it skew formal power series} over $R$; that is, all formal power series of $x$ with coefficients from $R$ with multiplication defined by $xr = \alpha(r)x$ for all $r \in R$. 

On the other hand, we also know that $$\dfrac{R[x,\alpha ]}{\langle x^{n}\rangle}\cong \dfrac{R[[x,\alpha ]]}{\langle x^{n}\rangle}.$$

We are now prepared to extract the following two consecutive statements.

\begin{corollary}\label{2.21}
Let $R$ be a ring with an endomorphism $\alpha$ such that $\alpha (1)=1$. Then, the following three assertions are equivalent:
\begin{enumerate}
\item
$R$ is a weakly strongly 2-nil-clean ring.
\item
$\dfrac{R[x,\alpha ]}{\langle x^{n}\rangle }$ is a weakly strongly 2-nil-clean ring.
\item
$\dfrac{R[[x,\alpha ]]}{\langle x^{n}\rangle }$ is a weakly strongly 2-nil-clean ring.
\end{enumerate}
\end{corollary}

\begin{corollary}
Let $R$ be a ring. Then, the following three claims are equivalent:
\begin{enumerate}
\item
$R$ is a weakly strongly 2-nil-clean ring.
\item
$\dfrac{R[x]}{\langle x^{n}\rangle }$ is a weakly strongly 2-nil-clean ring.
\item
$\dfrac{R[[x]]}{\langle x^{n}\rangle }$ is a weakly strongly 2-nil-clean ring.
\end{enumerate}
\end{corollary}

We are managed now to prove the following assertion on corner subrings. 

\begin{lemma}\label{2.14}
Let $R$ be a ring and $0\neq e=e^2\in R$. If $R$ is a weakly strongly 2-nil-clean ring, then so is $eRe$.
\end{lemma}

\begin{proof}
Choose $a \in eRe$. Thus, $a\in R$ and, thanks to Lemma \ref{three}, we have either  $a^3-a\in nil(R)$, or $a(a-1)(a-2)\in nil(R)$, or $a(a+1)(a+2)\in nil(R)$. As $a = ea = ae = eae$, we get $$a^3-a\in nil(R)\cap eRe\subseteq nil(eRe),$$ or $$a(a-1)(a-2)\in nil(R)\cap eRe\subseteq nil(eRe),$$ or $$a(a+1)(a+2)\in nil(R)\cap eRe\subseteq nil(eRe).$$  Hence, by Proposition \ref{three}, we deduce that $eRe$ is weakly strongly 2-nil-clean, as expected.
\end{proof}

The next criterion somewhat describes up to an isomorphism a partial variant of the weak strong 2-nil-cleanness.

\begin{lemma}\label{5} Let $R$ be a ring that $5\in nil(R)$. Then, $R$ is weakly strongly 2-nil-clean if, and only if, $R\simeq\mathbb{Z}_{5^k}$ for some positive integer $k$.
\end{lemma}

\begin{proof} $\Rightarrow $. In view of Proposition \ref{11}, $R/J(R)$ is weakly strongly 2-nil-clean and $J(R)$ is nil. Besides, Lemma \ref{2.12} employs to show that $R/J(R)$ is reduced. We now claim that $R/J(R)\simeq \mathbb{Z}_5$. To that goal, assume on the contrary that $$R/J(R)\neq \{J(R),1+J(R),2+J(R),-1+J(R),-2+J(R)\}.$$ Thus, there exists $a\in R$ such that $$a+J(R)\not\in  \{J(R),1+J(R),2+J(R),-1+J(R),-2+J(R)\}.$$ As $R/J(R)$ is weakly strongly 2-nil-clean, there are $e+J(R),f+J(R)\in Id(R/J(R))$ that all commute and $a+J(R)=\pm e+\pm f+J(R)$. Moreover, since $$a+J(R)\not\in  \{J(R),1+J(R),2+J(R),-1+J(R),-2+J(R)\},$$ we get either $$e+J(R)\not\in \{J(R),1+J(R),2+J(R),-1+J(R),-2+J(R)\},$$ or $$f+J(R)\not\in \{J(R),1+J(R),2+J(R),-1+J(R),-2+J(R)\}.$$ 

Firstly, if $e+J(R)\not\in \{J(R),1+J(R),2+J(R),-1+J(R),-2+J(R)\}$, then because any reduced ring is abelian, we receive $R/J(R)=(eR+J(R))/J(R)\times ((1-e)R+J(R))/J(R)$. Obviously, $5e,5(1-e)\in J(R)$ and so $(2e+J(R),1-e+J(R))\in R$. But, one checks that $$(2e+J(R),1-e+J(R))^3-(2e+J(R),1-e+J(R))\not\in J(R)$$ and $$(2e+J(R),1-e+J(R))^3\pm 3(2e+J(R),1-e+J(R))^2+2(2e+J(R),1-e+J(R))\not\in J(R),$$ leading to a contradiction. So, $$e+J(R)\in \{J(R),1+J(R),2+J(R),-1+J(R),-2+J(R)\}.$$ 

Analogically, we can see that $$f+J(R)\in \{J(R),1+J(R),2+J(R),-1+J(R),-2+J(R)\}.$$ This unambiguously means that $$R/J(R)=\{J(R),1+J(R),2+J(R),-1+J(R),-2+J(R)\}.$$ Finally, one discovers that $R/J(R)\simeq \mathbb{Z}_5$, as claimed. Hence, $R\simeq \mathbb{Z}_{5^k}$ for some integer $k$, as asserted.

$\Leftarrow$. It is rather clear, so we omit the argumentation.
\end{proof}

We end our work with the following attractive characterization of the weak strong 2-nil-clean property.

\begin{theorem}\label{maj} Let $R$ be a ring. Then, $R$ is weakly strongly 2-nil-clean if, and only if, $R\simeq R_1\times \mathbb{Z}_{5^k}$, where $R_1$ is a strongly 2-nil-clean ring and $k$ is a positive integer.
\end{theorem}

\begin{proof} Assume that $R$ is weakly strongly 2-nil-clean. Viewing Lemma \ref{three}, we have $30\in nil(R)$, and so we can derive that $R/30R$ is too weakly strongly 2-nil-clean. However, since $R/30R\simeq R/6R\times R/5R$, we consult with Corollary \ref{4} to get that $R_1$ is strongly $2$-nilclean ring and with Lemma \ref{5} to get that $R_2\simeq \mathbb{Z}_{5^k}$ for some integer $k>0$, as suspected.

Reciprocally, Lemma \ref{2.2} is a guarantor that the reverse implication is true.
\end{proof}

In closing, it is worthy of mentioning that the obtained decomposition in the last result also confirms that all weakly strongly 2-nil-clean rings are themselves strongly $\pi$-regular (compare with Lemma~\ref{2.12}).


\medskip



\vskip3.0pc

\begin{thebibliography}{99}
	
\bibitem{23}
G. C{\u{a}}lug{\u{a}}reanu, {\it UU-rings}, Carpathian J. Math. {\bf 31} (2) (2015), 157--163.		

\bibitem{chen}
H. Chen and M. Sheibani, {\it Strongly 2-nil-clean rings}, J. Algebra Appl. {\bf 16} (8) (2017).

\bibitem{chen2}
H. Chen and M. Sheibani, {\it Strongly weakly nil-clean rings}, J. Algebra Appl. {\bf 16} (12) (2017).

\bibitem{2}
A. J. Diesl, {\it Nil clean rings}, J. Algebra {\bf 383} (2013), 197--211.

\bibitem{12}
M. T. Kosan, {\it The P. P. property of trivial extensions}, J. Algebra Appl. {\bf 14} (2015).

\bibitem{14}
T. Y. Lam, A First Course in Noncommutative Rings, Second Edition, Graduate Texts in Math., Vol. {\bf 131}, Springer-Verlag, Berlin-Heidelberg-New York, 2001.

\bibitem{3}
A. R. Nasr-Isfahani, {\it On skew triangular matrix rings}, Commun. Algebra. {\bf 39} (11) (2011), 4461--4469.

\bibitem{16}
W. K. Nicholson, {\it Lifting idempotents and exchange rings}, Trans. Am. Math. Soc. {\bf 229} (1977), 269--278.

\bibitem{17}
W. K. Nicholson, {\it Strongly clean rings and fitting's lemma}, Commun. Algebra {\bf 27} (8) (1999), 2583--3592.

\bibitem{11}
G. Tang and Y. Zhou, {\it A class of formal matrix rings}, Linear Algebra Appl. {\bf 438} (2013), 4672--4688.

\bibitem{zhou}
Y. Zhou, {\it Rings in which elements are sums of nilpotents, idempotents and tripotents}, J. Algebra Appl. {\bf 17} (1) (2018).

\end{thebibliography}
\end{document}